\documentclass[11pt]{amsart}

\RequirePackage{amssymb}
\RequirePackage{enumerate}
\RequirePackage{amsfonts}
\RequirePackage{amsmath}
\RequirePackage{mathrsfs}
\allowdisplaybreaks[1]

\newcounter{fig}
\newcommand{\mybic}{\author{Gianluca Cassese}
                     \address{Universit\`{a} Milano Bicocca}
                     \email{gianluca.cassese@unimib.it}
                     \curraddr{Department of Economics, Statistics and Management 
                                  Building U7, Room 2097, via Bicocca 
                                  degli Arcimboldi 8, 20126 Milano - Italy}}

\newtheorem{theorem}{Theorem}
\theoremstyle{plain}

\newtheorem{corollary}{Corollary}

\newtheorem{lemma}{Lemma}

\newcommand{\tto}{\longrightarrow}

\newcommand{\Prob}{\mathbb{P}}

\newcommand{\Sim}{\mathscr{S}} 
\DeclareMathOperator*{\LIM}{LIM} 
\DeclareMathOperator*{\co}{co}

\newcommand{\A}{\mathscr{A}} 

\newcommand{\R}{\mathbb{R}} 
\newcommand{\N}{\mathbb{N}}

\newcommand{\Bor}{\mathscr{B}}

\newcommand{\abs}[1]{\vert #1\vert} 
\newcommand{\babs}[1]{\big\vert #1\big\vert}

\newcommand{\dabs}[1]{\left\vert #1\right\vert} 

\newcommand{\net}[3]{\langle #1_{#2}\rangle_{#2\in #3} } 
\newcommand{\neta}[1]{\net{#1}{\alpha}{\mathfrak A}} 
\newcommand{\nnet}[3]{\langle #1\rangle_{#2\in #3} } 
\newcommand{\nneta}[1]{\nnet{#1}{\alpha}{\mathfrak A}} 
\newcommand{\seq}[2]{\net{#1}{#2}{\mathbb{N}}} 
\newcommand{\sseq}[2]{\nnet{#1}{#2}{\mathbb{N}}} 
\newcommand{\seqn}[1]{\seq{#1}{n}} 
\newcommand{\sseqn}[1]{\sseq{#1}{n}}

\newcommand{\norm}[1]{\Vert #1\Vert} 
\newcommand{\bnorm}[1]{\big\Vert #1\big\Vert} 
\newcommand{\bgnorm}[1]{\bigg\Vert #1\bigg\Vert}

\newcommand{\condP}[2]{P(#1\vert #2)}

\newcommand{\set}[1]{\mathbf{1}_{#1}}

\newcommand{\cl}[2][ ]{\overline{#2}^{\ #1}}
\newcommand{\cco}[1][ ]{\overline\co^{#1}}

\newcommand{\G}[1]{\Gamma(#1_1,#1_2,\ldots)}

\newcommand{\Tkl}[1]{#1\wedge 2^k\lambda}
\newcommand{\Tl}[2]{#2\wedge#1\lambda}
\newcommand{\PP}{(\textit{\textbf P})\,}

\begin{document}

\title[Koml\'os Theorem]
{A Version of Koml\'os Theorem for Additive Set Functions}
\mybic
\date
\today
\subjclass[2010]{60F05, 60B10, 46B42.} 

\keywords{Koml\'os Lemma, Subsequence principle, Strong law of large 
numbers, Weak compactness.}

\maketitle

\begin{abstract}
We provide a version of the celebrated theorem of Koml\'os in which, 
rather then random quantities, a sequence of finitely additive measures 
is considered. We obtain a form of the subsequence principle and some 
applications.
\end{abstract}

{\SMALL The original publication is available at sankhya.isical.ac.in,
DOI: http://dx.doi.org/10.1007/s13171-015-0080-9.
}

\section{Introduction}

In 1967 Koml\'os \cite{komlos} proved the following subsequence principle: a 
norm bounded sequence $\seqn f$ in $L^1(P)$, with $P$ a probability law, 
admits a subsequence $\seqn g$ and $g\in L^1(P)$ such that, for any further 
subsequence $\seqn h$,
\begin{equation}
\label{intro}
P\left(\lim_{N\to\infty}\frac{h_1+h_2+\ldots+h_N}{N}= g\right)=1
\end{equation}
The proof uses a truncation technique, weak compactness in $L^2(P)$ 
and martingale convergence. In this paper we prove a form of this result
in which the random quantities $f_n$ are replaced by additive set functions
and countable additivity is not assumed.

The original work of Koml\'os has originated a number of subsequent 
contributions extending its validity in several directions. Chatterji \cite{chatterji} 
replaced $L^1$ with $L^p$ for $0<p<2$; Schwartz \cite{schwartz} gave 
two different proofs still using truncation and weak 
compactness; Berkes \cite{berkes} showed that the subsequence may be 
selected so that each permutation of its elements still satisfies \eqref{intro}. 
Other proofs of this same result (or some extension of it) were subsequently 
given also by Balder \cite{balder} and Trautner \cite{trautner}. Weizs\"acker 
\cite{weizsacker} explored the possibility of dropping the boundedness property 
while Lennard \cite{lennard} showed that this property is necessary for a convex 
subset of $L^1$ to have each sequence satisfying the subsequence principle.
Other papers considered cases in which the functions $f_n$ take their values in
some vector space other than $\R$. These include Balder \cite{balder 89} and 
Guessous \cite{guessous}. Balder and Hess \cite{balder hess} considered 
multifunctions with values in Banach spaces with the Radon Nikodym property. Day 
and Lennard \cite{day lennard} and Jim\'enez Fern\'andez et al. \cite{jimenez} 
proved equivalence with the Fatou property. Eventually, Halevy and Bhaskara Rao
\cite{halevy rao} considered the case in which the probability measure $P$ is 
replaced by an independent strategy, a finitely additive set function of a special
type introduced by Dubins and Savage \cite{dubins savage}.

Even disregarding the obvious interest in the strong law of large numbers, it is 
often very useful in applied problems to extract from a given sequence an a.s. 
converging subsequence. Koml\'os theorem implies that this may be done upon 
replacing the original sequence with one formed by convex combinations of 
elements of arbitrarily large index. In this somehow different formulation, Koml\'os 
theorem has been exploited extensively, e.g. by Burkholder \cite{burkholder} to 
give a simple proof of Kingman ergodic theorem, or by Cvitani\'c and Karatzas 
\cite{cvitanic karatzas} for application to statistics. 

However, consider replacing each element $f_n$ in the original sequence with 
 the indicator $\set B(f_n)$ of the event $f_n\in B$ and to construct the resulting 
empirical distribution:
\begin{equation}
\label{empirical}
F_N(B)=\frac{\sum_{n=1}^N\set B(f_n)}{N}
\qquad
B\in\Bor
\end{equation}
In order to apply Koml\'os to the time honored problem of the convergence of 
the empirical distribution, one should be able to select a subsequence so as to 
obtain convergence for all $B$ in $\Bor$. But this may hardly be possible if $\Bor$ 
is not countably generated, a situation quite common in the theory of stochastic 
processes when $\Bor$ is the Borel $\sigma$ algebra of some non separable 
metric space. This difficulty is emphasized if one requires a more sophisticated
notion of convergence than setwise convergence.

In a highly influential paper, Blackwell and Dubins \cite{blackwell dubins} modeled 
the evolution of probability in response to some observable phenomenon as a 
sequence of regular posterior probabilities:
\begin{equation}
\label{posterior}
F_n(B)=\condP{B}{f_1,\ldots,f_n}
\qquad
B\in\Bor
\end{equation}
The merging of opinions obtains whenever the posteriors originated by two 
countably additive probability measures converge to $0$ in total variation for 
all histories $f_1,f_2,\ldots$ save possibly on a set of measure zero. In this
formulation we are confronted with set functions taking values in a vector
space of measurable functions, a setting in which the subsequence principle
in its original formulation appears even more troublesome. 

The problem just considered provides a good case in point of the advantage 
or the need of working with finite additivity. On the one hand one may wish 
to define each $F_n$ in \eqref{posterior} on a larger class than $\Bor$, e.g. 
the class of all subsets of the underlying sample space $X$. Classical conditional 
expectation may then be extended fairly easily to this larger class (although 
not in a unique way), e.g. via \cite[Theorem 1]{JOTP}. On the other hand, 
one may consider this same problem in more general cases than with $X$ 
a complete, separable metric space so that the existence of a regular conditional 
probability may not be guaranteed. In either case, by virtue of the lifting theorem,
posterior probability may be defined as a vector valued, additive set function 
but the countable additivity property has to be abandoned.

The main result of this paper establishes that if a sequence of (finitely additive) 
probabilities is transformed by taking convex combinations and restrictions, then 
norm convergence obtains. The proof, although rather different from those given 
in the cited references, retains from the original work of Koml\'os, the idea of 
achieving weak compactness via truncations. In section \ref{sec banach} we prove 
the basic version of our result, valid for general Banach lattices with order continuous 
norm (but not assuming the Radon Nikodym property), as the key argument in our 
proof just uses lattice properties. With such degree of generality this result, perhaps 
of its own interest, appears to be rather weak. In section \ref{sec scalar} we 
specialize to the space $ba(\A)$ of scalar valued, additive, bounded set functions 
for which the convergence statement is significantly stronger. We provide some 
applications, such as a finitely additive version of the strong law of large numbers, 
Corollary \ref{cor slln}, and explore the implications of assuming independence, 
Corollary \ref{cor independent} and of dropping norm boundedness, Theorem 
\ref{th komlos unbounded}. In Theorem \ref{th komlos vector} we prove a version 
valid for a special space $ba_0(\A,X)$ of set functions taking values in some 
vector lattice $X$.

We refer throughout to a given, non empty set $\Omega$ and an algebra $\A$ 
of its subsets. $\Sim(\A)$  and $ba(\A)$ designate the families of simple, $\A$
measurable functions $f:\Omega\to\R$ and, as in \cite{bible}, the family of real 
valued, additive set functions on $\A$ which are bounded with respect to the total 
variation norm, $\norm\lambda=\abs\lambda(\Omega)$, respectively. When 
$\mu\in ba(\A)$ and $B\in\A$ we define $\mu_B\in ba(\A)$ implicitly by letting 
$\mu_B(A)=\mu(B\cap A)$ for  each $A\in\A$. $\Prob(\A)$ denotes the family 
of finitely additive probabilities on $\A$. Countable additivity is never assumed, 
unless otherwise explicitly stated. If $\lambda\in ba(\A)_+$, we say that a sequence 
$\seqn f$ of functions on $\Omega$ $\lambda$-converges to $0$ when
\begin{equation}
\label{l conv}
\lim_n\lambda^*(\abs{f_n}>\eta)=0
\qquad
\eta>0
\end{equation}
where, as usual, $\lambda^*$ denotes the outer measure
\begin{equation}
\lambda^*(B)=\inf_{\{A\in\A:B\subset A\}}\lambda(A)
\qquad
B\subset\Omega
\end{equation}
Likewise the expressions $\lambda$-Cauchy or $\lambda$-bounded refer
to the Cauchy property or to boundedness formulated relatively to the topology 
of $\lambda$-convergence.

\section{Banach Space Preliminaries.}
\label{sec banach}

The proof of the main theorem is based on the two technical results proved 
in this section in which $X$ is taken to be a real Banach space. If $K\subset X$ 
we write  $\cl K$ and $\cl[w] K$  to denote its closure  n the strong and in 
the weak topology, respectively, and $\co(K)$ for its convex hull. The symbols 
$\cco (K)$ and $\cco[w](K)$ will indicate the closed-convex hulls of $K$ in 
the corresponding topology. If $\seqn x$ is a sequence in $X$ the 
symbol $\G x$ will be used for the collection of those sequences $\seqn y$ 
in $X$ such that $y_n\in\co(x_n,x_{n+1},\ldots)$ for each $n\in\N$. 

We denote a given but arbitrary Banach limit on $\ell^\infty$ generically by 
$\LIM$. We
extend this notion to $X$ as follows: if $\seqn x$ is a norm bounded sequence 
in $X$, then there exists a unique $x^{**}\in X^{**}$ satisfying
\begin{equation}
\label{LIM def}
x^{**}(x^*)=\LIM_n(x^*x_n)
\qquad
x^*\in X^*
\end{equation}
and we write $\LIM_nx_n=x^{**}$. In the following we identify an element
of $X$ with its isomorphic image under the natural homomorphism
$\kappa:X\to X^{**}$ so that, when appropriate, we write $\LIM_nx_n\in X$.
Two properties follow easily from \eqref{LIM def}: (\textit{i})
$
\norm{\LIM_nx_n}
\le
\LIM_n\norm{x_n}
$
and (\textit{ii}) $x_n$ converges weakly to $x$ if and only if $\LIM_nx'_n=x$ 
for all subsequences $\seqn{x'}$%
\footnote{
Or, in yet other terms, if and only if all Banach limits on the original sequence
coincide with $x$.
}. 
A third one, less obvious, is the following
implication of Krein - \v Smulian theorem:

\begin{lemma}
\label{lemma LIM}
If $K$ is a relatively weakly compact  subset of a Banach space $X$, then 
\begin{equation}
\label{LIM}
\LIM_nx_n\in\bigcap_i\cco{}(x_i,x_{i+1},\ldots)
\quad\text{for every sequence}\quad
x_1,x_2,\ldots\in K
\end{equation}
\end{lemma}

\begin{proof}
Pick a sequence $\seqn x$ in $K$, write $K_i=\cco {}(x_i,x_{i+1},\ldots)$ and 
observe that
\begin{align*}
\inf_{x\in K_i}x^*(x)
	\le
(\LIM_nx_n)(x^*)
	\le
\sup_{x\in K_i}x^*(x)
\qquad
x^*\in X^*
\end{align*}
by the properties of the Banach limit and \eqref{LIM def}. The set $K_i$ is
convex and, by assumption and theorem \cite[V.6.4]{bible}, weakly
compact. It follows from a theorem of \v Smulian \cite[p. 464]{bible} that 
there exists $y_i\in K_i$ such that $x^*(y_i)=(\LIM_nx_n)(x^*)$ for each 
$x^*\in X^*$. In other words, $\LIM_nx_n\in\bigcap_i K_i$.
\end{proof}

Banach limits will be important in what follows but are used in other parts
of the theory of finitely additive set functions, e.g. to show the existence
of densities. Let us mention that this tool was also used by Ramakrishnan 
\cite{ramakrishnan} in the setting of finitely additive Markov chains.

A partially ordered, normed vector space $X$ is said to possess property 
\PP when every increasing, norm bounded net $\neta x$ in $X$ admits a 
least upper bound $x\in X$ and $\lim_\alpha\norm{x_\alpha-x}=0$.
Clearly, a normed vector lattice possessing property \PP is a complete 
lattice and its norm is order continuous, i.e. if $\neta x$ is an increasing
net in $X$ and if $x=\sup_\alpha x_\alpha\in X$ then 
$\lim_\alpha\norm{x-x_\alpha}=0$. Examples of Banach lattices with 
property \PP are the classical Lebesgue spaces $L^p$ as well as $ba(\A)$.

We recall that a Banach lattice $X$ is a vector lattice endowed with a norm such 
that $x,y\in X$ and $\abs x\le\abs y$ imply $\norm x\le\norm y$ and $X$ is norm 
complete. A Banach lattice with order continuous norm is complete as a lattice%
\footnote{
The proof of this claim is contained in that of \cite[12.9]{aliprantis burkinshaw}.
}.

\begin{lemma}
\label{lemma P}
Let $X$ be a Banach lattice possessing property \PP and denote by $ba_0(\A,X)$
the space of all finitely additive set functions $F:\A\to X$ endowed with the norm
\begin{equation}
\label{ba0}
\norm F_{ba_0(\A,X)}
	=
\sup_{\pi\in\Pi(\A)}\bigg\Vert\sum_{A\in\pi}\abs{F(A)}\bigg\Vert_X
\end{equation}
Then $ba_0(\A,X)$ is a Banach lattice with property \PP.
\end{lemma}

\begin{proof}
First of all it is clear that \eqref{ba0} defines a norm. Given our exclusive focus
on $ba_0(\A,X)$ in this proof, we shall use the symbol $\norm F$ in place of 
$\norm F_{ba_0(\A,X)}$. Let $\seqn F$ be a Cauchy sequence in 
$ba_0(\A,X)$. By \eqref{ba0} the sequence $\sseqn{F_n(A)}$ is Cauchy in 
$X$ and converges thus in norm to some limit $F(A)$, for each $A\in\A$. 
The set function implicitly defined $F:\A\to X$ is additive. Moreover,
\begin{align*}
\bigg\Vert\sum_{A\in\pi}\abs{(F-F_n)(A)}\bigg\Vert_X
	&\le
\sup_{r>n}\bigg\Vert\sum_{A\in\pi}\abs{(F_r-F_n)(A)}\bigg\Vert_X
	\le
\sup_{r>n}\norm{F_r-F_n}
\end{align*}
so that $\lim_n\norm{F-F_n}=0$ and 
$\norm F
	\le
\limsup_n\norm{F_n}$. 
$ba_0(\A,X)$ is thus a Banach space. We can introduce a partial order by 
saying that $F\ge G$ whenever $F(A)\ge G(A)$ for all $A\in\A$. Let $\neta F$ 
be a norm bounded, increasing net in $ba_0(\A,X)_+$. Fix $A\in\A$. Then 
$\nneta{F_\alpha(A)}$, an increasing, norm bounded net in $X$, converges 
in norm to some $F(A)\in X_+$ by the property \PP. Again $F$ is additive, 
$F\ge F_\alpha$ for all $\alpha\in\mathfrak A$ and
\begin{align*}
\bgnorm{\sum_{A\in\pi}F(A)}_X
	=
\lim_\alpha\bgnorm{\sum_{A\in\pi}F_\alpha(A)}_X
	\le
\sup_\alpha\norm{F_\alpha}
\end{align*}
In addition,
\begin{align*}
\bgnorm{\sum_{A\in\pi}\abs{(F-F_\alpha)(A)}}_X
	=
\bnorm{F(\Omega)-F_\alpha(\Omega)}_X
\end{align*}
so that $\lim_\alpha\norm{F-F_\alpha}=0$ and $ba_0(\A,X)$ possesses property \PP 
and its norm, as a consequence, is order continuous. It remains to show that it is a 
lattice and that the norm is a lattice norm, i.e. that $\abs F=\sup\{F,-F\}$ exists in 
$ba_0(\A,X)$ and that $\norm F=\norm{\abs F}$. 

Denote by $\Pi(\A)$ the collection of all partitions of $\Omega$ into finitely many
elements of $\A$. If $F\in ba_0(\A,X)$ and $\pi\in\Pi(\A)$ define the subadditive 
set function $F_\pi:\A\to X_+$ by letting
\begin{equation}
\label{Fpi}
F_\pi(A)
	=
\sum_{E\in\pi}\abs{F(A\cap E)}
\qquad
A\in\A
\end{equation}
Observe that $F_\pi\in ba_0(\A_\pi,X)$, where $\A_\pi\subset\A$ denotes the 
sub algebra generated by the partition $\pi\in\Pi(\A)$. The net $\net{F}{\pi}{\Pi(\A)}$ 
is increasing and norm bounded so that it converges in norm to some 
$F_*:\A\to X_+$ which is additive in restriction to $\A_\pi$ for all $\pi\in\Pi(\A)$,
i.e. $F_*\in ba_0(\A,X)$. Moreover, 
$F_*\ge\{F,-F\}$. Any $G\in ba_0(\A,X)$ with $G\ge\{F,-F\}$ is also such that
$G(A)=\sum_{E\in\pi}G(A\cap E)\ge\sum_{E\in\pi}\abs{F(A\cap E)}=F_\pi(A)$
and so $G\ge F_*$. This proves that $\abs F=F_*$ and thus that $ba_0(\A,X)$
is a vector lattice. To see that its norm is a lattice norm observe that
$\norm{F_*}
	=
\norm{F_*(\Omega)}_X
	=
\lim_\pi\norm{F_\pi(\Omega)}_X
	=
\norm F
$.
\end{proof}

The norm introduced on $ba_0(\A,X)$ differs from the variation and semivariation 
norms usually considered for vector measures. It seems to be appropriate
for the somehow unusual case in which the set functions take value in a vector
space endowed with a lattice structure. A nice consequence of the lattice property 
is the relative ease of the weak compactness condition, compared to general spaces 
of vector measures, see \cite{brooks} and \cite{brooks dinculeanu}, and the nice 
interplay between norm and order which is crucial to our approach.

The following lattice inequalities will be useful:
\begin{subequations}
\label{lattice}
\begin{equation}
\label{lattice a}
( x+ y)\wedge z
	\le
( x\wedge z)+( y\wedge z)
\qquad x,y,x\in X_+
\end{equation}
\begin{equation}
\label{lattice b}
\abs{x\wedge z-y\wedge z}\le\abs{x-y}
\qquad
x,y,z\in X
\end{equation}
\end{subequations}

\begin{theorem}
\label{th komlos lattice}
Let $X$ be a Banach lattice with order continuous norm and $\seqn x$ a 
sequence in $X_+$. Fix $z\in X_+$. There exist three sequences in $X_+$, 
(i) $\seqn y$ in $\G x$, 
(ii) $\seqn\zeta$ with $\zeta_n\le y_n$ for $n=1,2,\ldots$ and 
(iii) $\seqn\xi$ increasing, 
such that
\begin{equation}
\label{komlos lattice}
\lim_n\bnorm{\zeta_n\wedge 2^kz-\xi_k}=0
\quad\text{and}\quad
y_n\wedge 2^kz\xrightarrow{\makebox[1.2cm]{\tiny{weakly}}} \xi_k
\qquad\text{for all}\quad
k\in\N
\end{equation}
\end{theorem}

This theorem proves that any positive sequence may be suitably transformed
to obtain some form of convergence via convexification and truncation. The
important fact is that \textit{the same} convex sequence, $\seqn y$, possesses
the weak convergence property \textit{for any} truncation adopted. This
delicate property is obtained exploiting the order structure of Banach lattices
with order continuous norm. If we replace the family of convex sequences 
by those sequences which are dominated by an element of such family,
then we obtain norm convergence.

\begin{proof}
Fix the following families:
\begin{align}
\label{C}
\mathscr C(n)
	=
\big\{u\in X_+:u\le u'\text{ for some }u'\in\co(x_n,x_{n+1},\ldots)\big\}
	\quad\text{and}\quad
\mathscr C=\bigcap_n\cl{\mathscr C(n)}
\end{align}
and notice that $x\in\mathscr C$ implies $\norm x\le\limsup_n\norm{x_n}$
and, by \eqref{lattice b}, $x\wedge u\in\mathscr C$ for all $u\in X$.

In a Banach lattice all sets admitting a lower as well as an upper bound
are relatively weakly compact, \cite[Theorem 12.9]{aliprantis burkinshaw}.
Thus, by Lemma \ref{lemma LIM}, for every sequence $\seqn u$ 
in $\G x$
\begin{equation}
\label{in C}
\LIM_n\big(u_n\wedge 2^kz\big)
	\in
\bigcap_n\cco\big(u_n\wedge 2^kz,u_{n+1}\wedge 2^kz,\ldots\big)
	\subset
\mathscr C
\end{equation}
Let $\Xi$ designate the family of all sequences $\tilde y=\seqn y$ in $\mathscr C$ 
with $y_{n-1}\le y_n\le 2^nz$. Let $\Xi$ be partially ordered by the 
product order and let $\Xi_0=\{\tilde y^\alpha:\alpha\in\mathfrak A\}$ be a 
chain in $\Xi$. Then, $\{y^\alpha_n:\alpha\in\mathfrak A\}$ is a chain in 
$\mathscr C$ admitting $2^nz$ as an upper bound in $X$. Since $X$ is a 
complete lattice, $y_n=\sup_\alpha y^\alpha_n$ exists in $X$ and, by order 
continuity of the norm, in $\mathscr C$. Of course, $y_{n-1}\le y_n\le 2^nz$
so that $\tilde y=\seqn y$ is an upper bound for $\Xi_0$. By Zorn's lemma, 
$\Xi$ admits a maximal element which we denote by $\tilde\xi=\seqn\xi$.

If $j>0$  and $\xi^j_n=\xi_{n+j}\wedge 2^nz$, then $\seqn{\xi^j}$ is an 
element of $\Xi$ dominating $\tilde\xi$. Thus,
\begin{equation}
\label{restr}
\xi_n\wedge 2^kz=\xi_k
\qquad
n\ge k
\end{equation} 
By the inclusion $\xi_k\in\mathscr C$, there exist two sequences $\seq \zeta k$ 
and $\seq yk$ such that $0\le\zeta_k\le y_k\in \co(x_k,x_{k+1},\ldots)$ and 
$\norm{\xi_k-\zeta_k}<2^{-k}$. It follows from \eqref{lattice b} and \eqref{restr}
that 
$
\abs{\xi_k-\zeta_n\wedge 2^kz}
	=
\abs{\xi_n\wedge2^kz-\zeta_n\wedge 2^kz}
	\le
\abs{\xi_n-\zeta_n}
$
 and so
\begin{align}
\label{dom}
\xi_k
	=
\lim_n\zeta_n\wedge 2^kz
	\le
\LIM_n\big(y_n\wedge 2^kz\big)
	\equiv
\hat y_k
\end{align}
Thus, by \eqref{in C}, $\seq{\hat y}{k}$ is yet another sequence in 
$\Xi$ dominating $\seqn\xi$ so that $\hat y_k=\xi_k$.  Since \eqref{dom} 
holds for any subsequence, we obtain that $y_n\wedge 2^kz$ converges to 
$\xi_k$ weakly for every $k\ge 0$.
\end{proof}

We stress that this result does not assume norm boundedness of the original 
sequence $\seqn x$.

\section{Koml\'os Theorem for Additive Set Functions.}\label{sec scalar}

In this section we apply Theorem \ref{th komlos lattice} to $X=ba(\A)$. 
The following is the main result of the paper.

\begin{theorem}
\label{th komlos}
Let $\seqn F$ be a norm bounded sequence in $ba(\A)_+$, $\delta>0$ and 
$\lambda\in\Prob(\A)$. There exist (a) $\xi\in ba(\lambda)_+$ with 
\begin{equation}
\label{lower bound}
\norm\xi
	\ge
\sup_k\limsup_n\bnorm{F_n\wedge 2^k\lambda}-\delta
\end{equation}
(b) $\seqn G$ in $\G F$ and
(c) $\seqn A$ in $\A$ such that, letting $\bar G_n=G_{n,A_n}$,
\begin{equation}
\label{komlos}
\lim_n\bnorm{\bar G_n-\xi}
	=
0
\quad\text{and}\quad
\sum_n\lambda(A^c_n)<\infty
\end{equation}
Moreover, the following are equivalent: 
(i)
$\xi=0$ is the only choice that satisfies \eqref{komlos} for some $\lambda$,
$\seqn G$ and $\seqn A$ as above;
(ii)
the sequence $\seqn F$ is asymptotically orthogonal, i.e. 
\begin{equation}
\label{orth asy}
\lim_n\bnorm{F_n\wedge F_j}
	=
0
\qquad\text{for}\quad
j=1,2,\ldots
\end{equation}
\end{theorem}

\begin{proof}
Let $\eta=\lim_k\limsup_n\norm{F_n\wedge2^k\lambda}-\delta$. Passing to 
a subsequence, we assume with no loss of generality that 
$\lim_k\liminf_n\norm{\abs{F_n}\wedge2^k\lambda}>\eta$.

Since $ba(\A)$ is a Banach lattice with order continuous norm, we can invoke 
Theorem \ref{th komlos lattice} with $F_n$, $G_n$ and $H_n$ in place of $x_n$, 
$y_n$ and $\zeta_n$ respectively. Then $\Tkl{G_n}$ and $\Tkl{H_n}$ converge 
weakly to $\xi_k$ but $\Tkl{G_n}\ge\Tkl{H_n}$. This implies that 
$\Tkl{G_n}-\Tkl{H_n}$ converges to $0$ in norm and therefore that 
\begin{equation}
\label{Tkl}
\lim_n\bnorm{\Tkl{G_n}-\xi_k}=0
\qquad
k\in\N
\end{equation}
As $\seq\xi k$ is increasing and norm bounded, property \PP implies that it 
converges in norm to some $\xi\ll\lambda$. Upon passing to a subsequence 
we deduce 
\begin{equation}
\label{komlos 1}
\lim_n\bnorm{G_n\wedge2^n\lambda-\xi}=0
\end{equation}
and in turn
\begin{align*}
\norm{\xi}
	=
\lim_n\bnorm{\Tl {2^n}{G_n}}
	\ge
\lim_k\liminf_n\bnorm{\Tl{2^k}{F_n}}
	>
\eta
\end{align*}
Moreover, selecting a further subsequence if necessary, we can assume that 
the sequence $\seqn\alpha$ of convex weights associated with $\seqn G$
via $G_n=\sum_i\alpha_{n,i}F_i$ is disjoint, i.e. $\alpha_{n,i}\alpha_{m,i}=0$ 
when $n\ne m$.

Choose $A_n\in\A$ such that 
$G_n(A_n)+2^n\lambda(A_n^c)
	\le
\norm{G_n\wedge2^n\lambda}+2^{-n}$
and observe that 
\begin{equation}
\label{lim xi}
\sum_{j\ge n}\lambda(A_j^c)
	\le
\sum_{j\ge n}2^{-j}\big[\norm{G_j\wedge2^j\lambda}+2^{-j}\big]
	\le
2^{-n}\Big(1+\sup_i\norm{F_i}\Big)
\end{equation}
Let $\bar G_n=G_{n,A_n}$ and $\hat G_n=\bar G_n+2^n\lambda_{A_n^c}$.
It follows that $\hat G_n\ge G_n\wedge2^n\lambda$ and therefore
$$
\norm{\hat G_n-G_n\wedge2^n\lambda}
	=
\norm{\hat G_n}-\norm{G_n\wedge2^n\lambda}
	=
G_n(A_n)+2^n\lambda(A_n^c)-\norm{G_n\wedge2^n\lambda}
	\le
2^{-n}
$$
$\hat G_n$ converges thus in norm to $\xi$ and we conclude that
\begin{align*}
\norm{\bar G_n-\xi}
	\le
\abs{\hat G_n-\xi}(A_n)+\xi(A_n^c)
	\le
\norm{\hat G_n-\xi}+\xi(A_n^c)
\end{align*}
Thus, \eqref{komlos} follows easily from \eqref{lim xi} and absolute continuity.

Suppose that (\textit{i}) holds. Then it must be that 
$\limsup_n\norm{F_n\wedge \mu}=0$ for each $\mu\in ba(\A)_+$, 
including $F_j$ for $j=1,2,\ldots$. Conversely, assume (\textit{ii}) fix 
$\lambda\in\Prob(\A)$ and let 
$H=\sum_n2^{-n}F_n$. By induction it is easily established the decomposition
$
\lambda
	=
\sum_{j=0}^\infty\lambda_j^\perp
$
where $F_0\equiv\lambda_0^\perp\perp H$ while
$F_j\gg\lambda_j^\perp\perp\{F_0,\ldots,F_{j-1}\}$
for
$j\ge1$.
Observe that, by orthogonality, $\norm\lambda=\sum_j\norm{\lambda_j^\perp}$. 
Moreover, for fixed $k,\varepsilon,j>0$ there exists $t>1$ such that, by \eqref{lattice a},
\begin{equation}
\label{pippo}
\begin{split}
\bigg\Vert{F_n\wedge2^k\Big(\sum_{0\le i\le j}\lambda_i^\perp\Big)}\bigg\Vert
	\le
\varepsilon+\bgnorm{F_n\wedge t\sum_{1\le i\le j}F_i}
	\le
\varepsilon+t\sum_{1\le i\le j}\norm{F_n\wedge F_i}
\end{split}
\end{equation}
We conclude that 
\begin{equation*}
\limsup_n\bnorm{G_n\wedge2^k\lambda}
	=
\lim_j\limsup_n\bigg\Vert{G_n\wedge2^k\sum_{i>j}\lambda_i^\perp}\bigg\Vert
	\le
2^k\lim_j\sum_{i>j}\bnorm{\lambda_i^\perp}
	=
0
\end{equation*}
and thus that 
$\norm\xi
	=
\lim_k\bnorm{\xi\wedge2^k\lambda}
	=
\lim_k\lim_n\bnorm{G_n\wedge2^k\lambda}
	=
0$. 
\end{proof}

It is possible to drop the assumption that the sequence $F_1,F_2,\ldots$ is positive
although at the cost of loosing some information on the limit $\xi$.

\begin{corollary}
\label{cor komlos sign}
Let $\seqn F$ be a norm bounded sequence in $ba(\A)$ and 
$\lambda\in\Prob(\A)$. There exist 
(a) $\xi\in ba(\lambda)$, 
(b) $\seqn G$ in $\G F$ and
(c) $\seqn A$ in $\A$ such that, letting $\bar G_n=G_{n,A_n}$,
\begin{equation}
\label{komlos sign}
\lim_n\bnorm{\bar G_n-\xi}
	=
0
\quad\text{and}\quad
\sum_n\lambda(A^c_n)<\infty
\end{equation}
\end{corollary}

\begin{proof}
From Theorem \ref{th komlos} we conclude that 
$\lim_n\norm{H_{n,B_n}-\chi}=0$ 
where $\seqn H$ is a sequence in $\G{F^+}$, $\seqn B$ a sequence in 
$\A$ with $\sum_n\lambda(B_n^c)<\infty$ and $\chi\in ba(\lambda)_+$.
Let $\seqn\alpha$ be the disjoint sequence of convex weights associated with $H_n$
via
\begin{equation*}
H_n=\sum_i\alpha_{n,i}F_i^+
\qquad
n\in\N
\end{equation*}
Write $\bar F_n=\sum_i\alpha_{n,i}F_i^-$ and apply Theorem \ref{th komlos} 
to $\seqn{\bar F}$. We obtain a disjoint sequence $\seq \beta k$ of 
convex weights, a sequence $\seqn C$ in $\A$ and $\zeta\in ba(\lambda)_+$ 
such that, letting
$K_j	=\sum_n\beta_{j,n}\bar F_n$,
\begin{align*}
\lim_j\norm{K_{j,C_j}-\zeta}=0
\quad\text{and}\quad
\sum_j\lambda(C_j^c)<\infty
\end{align*}
Let 
$
\gamma_{j,i}=\sum_n\beta_{j,n}\alpha_{n,i}
$, 
$
G_j
	=
\sum_i\gamma_{j,i}F_i
$
 and
$
A_j=C_j\cap\bigcap_{\{n:\beta_{j,n}>0\}}B_n\in\A
$. 
Observe that $G_j\in\co(F_j,F_{j+1},\ldots)$. Moreover, since 
$\beta_{j,n}\beta_{j',n}=0$ when $j\ne j'$,
\begin{align*}
\sum_j\lambda(A_j^c)
	=
\sum_j\lambda(C_j^c)+\sum_j\sum_{\{n:\beta_{j,n}>0\}}\lambda(B_n^c)
	\le
\sum_j\lambda(C_j^c)+\sum_n\lambda(B_n^c)
	<
\infty
\end{align*}
But then
$
G_j
	=
\sum_n\beta_{j,n}H_n
	-
K_j
$
so that, letting $\xi=\chi-\zeta$,
\begin{align*}
\norm{G_{j,A_j}-\xi}
	\le
\abs{\chi(A_j^c)}+\abs{\zeta(A_j^c)}
	+
\sum_n\beta_{j,n}\norm{H_{n,A_j}-\chi_{A_j}}
	+
\norm{K_{j,A_j}-\zeta_{A_j}}
\tto0
\end{align*}
\end{proof}

A few comments are in order.

(1). Assume that the original sequence $\seqn F$ is relatively weakly compact.
It is then uniformly absolutely continuous with respect to some $\lambda$ (see 
\cite[Theorems 2.3 and 4.1]{brooks dinculeanu}) so that 
$\lim_k\sup_n\norm{G_n-G_n\wedge2^k\lambda}=0$ 
and the sequence $\seqn G$ converges thus strongly to $\xi$. In this special case, 
Theorem \ref{th komlos} is nothing more than the classical result of Banach and
Saks \cite[III.3.14]{bible}. In the general case, however, the sequence $\seqn G$ 
need not converge to $\xi$, not even in restriction to a single, fixed set. If one 
assumes that $\A$ is a $\sigma$ algebra and $\lambda$ is countably additive, 
then it becomes possible to replace the sets $A_1,A_2,\ldots$ with 
$A_k^*=\bigcap_{n>k}A_n$ 
and conclude that for each $\varepsilon$ there is $A\in\A$ such that 
$\lambda(A^c)<\varepsilon$ while $G_n$ converges in norm to $\xi$ in restriction 
to $A$. This improvement on the statement of Theorem \ref{th komlos} may also
be obtained upon introducing a form of the independence property valid in the finitely 
additive context. The first step in this direction was made long ago by Purves and 
Sudderth \cite{purves sudderth} relatively to strategies, a notion due to Dubins and 
Savage \cite{dubins savage} and too lengthy to explain here. Briefly put, an 
independent strategy is a finitely additive probability $\lambda$ defined over the
algebra of clopen sets of the product space $X^\N$, with $X$ an arbitrary non void 
set, and satisfying
\begin{equation}
\lambda(H_1\times H_2\times\ldots H_N\times\ldots)
	=
\gamma_1(H_1)\gamma_2(H_2)\ldots\gamma_N(H_N)\ldots
\end{equation}
where $\seqn H$ is a sequence of subsets of $X$ and $\seqn\gamma$
a sequence of finitely additive probabilities defined on all subsets of $X$.
This notion has found a number of applications in replicating finitely additive
theorems on the convergence of random quantities. Given that our interest 
focuses instead on additive functions, we think that the following is a reasonable 
adaptation of that same notion to our setting.

Say that a sequence $\seqn F$ in $ba(\A)$ is independent relatively to 
$\lambda\in ba(\sigma\A)$ if $\lambda\gg F_n$ for $n=1,2,\ldots$ and there
exists a sequence $\seqn\A$ of subalgebras of $\A$ with the property that
(\textit{i}) $\inf_{h\in\Sim(\A_n)}\norm{F_n-\lambda_h}=0$ and (\textit{ii})
for any countable partition $\{N_1,N_2,\ldots\}$ of $\N$ into finite subsets
\begin{equation}
\label{indep}
\lambda\Big(\bigcap_i B_i\Big)
	=
\prod_{i=1}^\infty\lambda(B_i)
\qquad
B_i\in\bigvee_{n\in N_i}\A_n,\ i=1,2,\ldots
\end{equation}

We state the following corollary only for the case of positive set functions.

\begin{corollary}
\label{cor independent}
Let $\seqn F$ be a bounded sequence in $ba(\sigma\A)_+$ which is independent 
relatively to $\lambda\in\Prob(\sigma\A)$ and let $\delta>0$. There exist 
$\xi\in ba(\A)_+$ satisfying 
$\norm{\xi}\ge\sup_k\limsup_n\bnorm{F_n\wedge2^k\lambda}-\delta$, and a 
sequence $\seqn G$ in $\G F$ such that, for each $\varepsilon>0$, there is 
$A_\varepsilon\in\A$ such that
\begin{equation}
\xi(A_\varepsilon^c)<\varepsilon
\quad\text{and}\quad
\lim_n\abs{G_n-\xi}(A_\varepsilon)=0
\end{equation}
\end{corollary}

\begin{proof}
Choose $f_n\in\Sim(\A_n)$ such that $\norm{F_n-\lambda_{f_n}}<2^{-n-1}$. 
Let $\seqn G$ be the sequence in \eqref{komlos} with $G_n=\sum_i\alpha_{n,i}F_i$ 
and write $g_n=\sum_i\alpha_{n,i}f_i$. Choose the sequence $\seqn\alpha$
to be disjoint. Observe that $\norm{G_n-\lambda_{g_n}}\le2^{-n-1}$. 
For each $A\in\A$
\begin{align*}
G_n(A)+2^n\lambda(A^c)
	&\ge
-2^{-n-1}
+\lambda_{g_n}(A)+2^n\lambda(A^c)\\
	&\ge
-2^{-n-1}+\lambda(g_n\wedge2^n)\\
	&=
-2^{-n-1}+\lambda_{g_n}\big(g_n\le2^n\big)+2^n\lambda\big(g_n>2^n\big)\\
	&\ge
-2^{-n}+G_n\big(g_n\le2^n\big)+2^n\lambda\big(g_n>2^n\big)
\end{align*}
so that 
$G_n(g_n\le2^n)+2^n\lambda(g_n>2^n)
	\le
2^{-n}+\norm{G_n\wedge2^n\lambda}$.
Thus we can replace $A_n$ with $\{g_n\le2^n\}$ in Theorem 
\ref{th komlos} and obtain
\begin{align*}
\sum_{n>N}\lambda\big(g_n>2^n\big)\le2^{-N}\Big(1+\sup_n\norm{F_n}\Big)
\end{align*}
Given that the sets $N_n=\{i:\alpha_i^n\ne0\}$ are disjoint and that
$g_n\in\Sim\big(\bigvee_{i\in N_n}\A_i\big)$ we conclude, following the
 classical proof of the Borel-Cantelli lemma,
\begin{align*}
\lambda\bigg(\bigcap_{n>N}\{g_n\le2^n\}\bigg)
	=
\prod_{n>N}\lambda(g_n\le2^n)
	\ge 
1-\sum_{n>N}\lambda\big(g_n>2^n\big)
	\ge
1-2^{-N}\Big(1+\sup_n\norm{F_n}\Big)
\end{align*}
i.e. $\lim_N\lambda\big(\bigcap_{n>N}\{g_n\le2^n\}\big)=1$ and, by absolute
continuity, $\lim_N\xi\big(\bigcup_{n>N}\{g_n>2^n\}\big)=0$. One can then fix 
$A_\varepsilon=\bigcap_{n>N}\{g_n\le2^n\}\in\A$ with $N$ sufficiently large. 
The claim follows from \eqref{komlos}.
\end{proof}

(2). Notice the special case in which $F_n=\int f_nd\lambda$ with $\seqn f$ a 
bounded sequence in $L^1(\lambda)$ so that $G_n$ takes the form 
$G_n=\int g_nd\lambda$ for some $g_n\in\co(f_n,f_{n+1},\ldots)$. Then, with 
the notation of Corollary  \ref{cor komlos sign},
\begin{align*}
\lambda^*(\abs{g_n-g_m}>c)
	&\le
\lambda^*(\abs{g_n-g_m}\set{A_n\cap A_m}>c)
	+
\lambda(A_n^c\cup A_m^c)\\
	&\le
c^{-1}\int_{A_n\cap A_m}\abs{g_n-g_m}d\lambda+\lambda(A_n^c\cup A_m^c)\\
	&\le
c^{-1}(\norm{\bar G_{n,A_m}-\xi}+\norm{\bar G_{m,A_n}-\xi})
	+
\lambda(A_n^c\cup A_m^c)
\end{align*}
so that the sequence $\seqn g$ is $\lambda$-Cauchy, a claim proved in
\cite[Theorem 6.3]{BK} for the case $f_n\ge0$. If, in addition, $\lambda$
is countably additive, then by completeness we obtain that $g_n$ 
$\lambda$-converges to some limit $h\in L^1(\lambda)$ or even
converges a.s., upon passing to a subsequence. This is the form in which 
the subsequence principle attributed to Koml\'os is often stated in applications. 

(3). 
In a possible interpretation of Theorem \ref{th komlos}, one may take $F_n$
to be an expression of the disagreement $\abs{F^1_n-F^2_n}$ between two
different opinions. Theorem \ref{th komlos} suggests that either disagreement 
is progressively smoothed out, in accordance with condition \eqref{orth asy}, or 
that it converges to some final divergence of opinions. It would be interesting
to see if this result may be useful to get more insight in the classical problem
of merging of opinions described in the well known paper of Blackwell and Dubins 
\cite{blackwell dubins}.

We close this section with two generalizations of Theorem \ref{th komlos}.
In the first we drop the norm boundedness condition; in the second one we 
consider vector valued set function.

\begin{theorem}
\label{th komlos unbounded}
Let $\seqn F$ be a sequence in $ba(\A)_+$ and $\lambda\in\Prob(\A)$.
There exist 
$\xi:\A\to\R_+\cup\{\infty\}$ finitely additive and 
a sequence $\seqn G$ in $\G F$ such that
\begin{equation}
\lim_n\babs{G_n\wedge2^n\lambda-\xi}(A)=0
\qquad
\text{for each }A\in\A\text{ with }\xi(A)<\infty
\end{equation}
\end{theorem}

\begin{proof}
Given that Theorem \ref{th komlos lattice} does not require norm boundedness,
the sequence $\seq\xi k$ is obtained exactly as in the proof of Theorem 
\ref{th komlos}. Define the extended real valued, finitely additive set 
function
\begin{equation}
\xi(A)=\lim_k\xi_k(A)
\qquad
A\in\A
\end{equation}
and notice that \eqref{Tkl} holds so that the sequence $\seqn G$ may be 
chosen in such a way that $\norm{G_n\wedge2^n\lambda-\xi_n}<2^{-n}$.
Thus, if $A\in\A$ and $\xi(A)<\infty$,
\begin{align*}
\lim_n\babs{\xi-G_n\wedge2^n\lambda}(A)
	=
\lim_n\abs{\xi-\xi_n}(A)%\\
	=
\lim_n\sup_{\pi\in\Pi(\A)}\sum_{B\in\pi}\abs{(\xi-\xi_n)(A\cap B)}%\\
	=
\lim_n(\xi-\xi_n)(A)%\\
	=0
\end{align*}
\end{proof}

Spaces such as $L^1$ or $ba$ have the special property that a sequence
of positive elements converges weakly to $0$ if and only if it converges in 
norm too. For this special spaces we obtain the following:

\begin{theorem}
\label{th komlos vector}
Let $(W,\Bor,P)$ be a classical probability space -- i.e. $W$ non empty, $\Sigma$
a $\sigma$ algebra of subsets of $W$ and $P$ $\sigma$ additive -- and set 
$X=L^1(W,\Bor,P)$. Let $\seqn F$, a norm bonded sequence in $ba_0(\A,X)_+$, 
and $\lambda\in\Prob(\A)$ be such that the set
\begin{equation}
\label{bdd}
\mathcal R
	=
\co\Big\{\sup_{A\in\pi}F_n(A)/\lambda(A):n\in\N,\ \pi\in\Pi(\A)\Big\}
\end{equation}
is $P$-bounded. There exist $\xi\in ba_0(\A,X)_+$ and $\seqn G$ in 
$\G F$ such that
\begin{equation}
\label{komlos vector}
\lim_n\babs{G_n-\xi}(\Omega)=0
\qquad
P-a.s.
\end{equation}
\end{theorem}

\begin{proof}
Given that $X$ is a Banach lattice possessing property \PP and that weak and
strong convergence to $0$ are equivalent properties for positive sequences in
$X$, we deduce from Theorem \ref{th komlos lattice} that there exists 
$\xi\in ba_0(\A,X)_+$ and a sequence $\seqn G$ in $\G F$ such that
\begin{equation}
\lim_n\bnorm{G_n\wedge2^n\lambda-\xi}=0
\end{equation}
Observe that for each $A\in\A$,
\begin{align*}
(G_n\wedge2^n\lambda)(A)
	=
\lim_\pi\sum_{A'\in\pi}G_n(A\cap A')\wedge2^n\lambda(A\cap A')
\end{align*}
and, since the net on the right hand side is decreasing with $\pi$, 
there exists $\pi_n=\{A_n^1,\ldots,A_n^{I_n}\}\in\Pi(\A)$ such that 
\begin{align*}
(G_n\wedge2^n\lambda)(A)
	\le
\sum_{i=1}^{I_n}G_n(A\cap A_n^i)\wedge2^n\lambda(A\cap A_n^i)
	\le
\sum_{i=1}^{I_n}G_n(A\cap A_n^i)\set{B_n^i}+2^n\lambda(A\cap A_n^i)\set{B_n^{ic}}
	\equiv
\bar G_n(A)
\end{align*}
(the last equality being a definition of $\bar G_n\in ba_0(\A,X)$) where 
\begin{equation*}
B_n^i=\big\{G_n(A_n^i)\le2^n\lambda(A_n^i)\big\}
\end{equation*}
On the other hand,
\begin{align*}
\norm{\bar G_n-G_n\wedge2^n\lambda}
	&=
\norm{\bar G_n(\Omega)}_X-\norm{(G_n\wedge2^n\lambda)(\Omega)}_X\\
	&=
\int{\sum_{i=1}^{I_n}G_n(A_n^i)\wedge2^n\lambda(A_n^i)}dP
-
\lim_\pi\int{\sum_{A\in\pi}G_n(A)\wedge2^n\lambda(A)}dP
\end{align*}
so that $\pi_n$ may be so chosen that 
$\norm{\bar G_n-G_n\wedge2^n\lambda}\le2^{-n}$.
We conclude, $\lim_n\norm{\bar G_n-\xi}=0$. Observe that, with the above notation,
\begin{equation*}
B_n
	=
\bigcap_{i=1}^{I_N}B_n^i
	=
\Big\{\sup_{A\in\pi_n}G_n(A)/\lambda(A)\le2^n\Big\}
\end{equation*}
Given that $\sup_{A\in\pi_n}G_n(A)/\lambda(A)\in\mathcal R$ then, by assumption, 
$\lim_nP(B_n)=1$. Passing to a subsequence (still indexed by $n$) we obtain that 
$\sum_nP(B_n^c)<\infty$ and therefore that for each $k>0$ there exists $N_k>N_{k-1}$ 
such that 
\begin{equation}
P\Big(\bigcap_{n>N_k}B_n\Big)>1-\varepsilon
\end{equation}
Let $H_k=\bigcap_{n>N_k}B_n$.
\begin{align*}
\int_{H_k}\abs{G_n-\xi}(\Omega)dP
	&=
\lim_\pi\int_{H_k}\sum_{A\in\pi}\abs{\bar G_n(A_n)-\xi(A_n)}dP
	\le
\norm{\bar G_n-G_n}+\norm{G_n\wedge2^n\lambda-\xi}
\end{align*}
We obtain a subsequence such that $\abs{G_n-\xi}(\Omega)$ converges
to $0$ $P$ a.s. on $H_k$ for each $k$. Given that the sequence $\seq Hk$
is increasing, we conclude that the sequence convergence pointwise
on $H=\bigcup_kH_k$ and that $P(H)=1$.
\end{proof}

\section{Further Applications}

A first, simple application of Theorem \ref{th komlos} is the following:

\begin{corollary}
Let $\seqn\mu$ be a norm bounded sequence in $ba(\A)$ and let
$K_n\subset K_{n+1}\cap L^1(\mu_n)$ for $n=1,2,\ldots$.
Write $K=\bigcup_nK_n$ and assume that
\begin{equation}
\sup_k\limsup_n\bnorm{\abs{\mu_n}\wedge\abs{\mu_k}}>0
\qquad\text{and}\qquad
\limsup_n\norm{f}_{L^1(\mu_n)}<\infty
\qquad 
f\in K
\end{equation}
Then there is $\mu\in\Prob(\A)$ such that $K\subset L^1(\mu)$.
\end{corollary}	

\begin{proof}
The statement remains unchanged if we replace $\mu_n$ with $\abs{\mu_n}$
so that we can assume with no loss of generality that $\mu_n\ge0$ for all $n\in\N$. 
Theorem \ref{th komlos} delivers the existence of $\bar\mu\in ba(\A)_+$ 
such that $\norm{\bar\mu}>0$ and $\lim_n\norm{\bar\mu-m_n\wedge2^n\lambda}=0$ 
for some $\lambda\in ba(\A)_+$ and $\seqn m$ in $\G\mu$. Write 
$\mu=\bar\mu/\norm{\bar\mu}\in\Prob(\A)$. Let $f\in K$ and for each $n\in\N$ 
sufficiently large, let $\seq{f^n}{j}$ be a sequence of $\A$ simple functions 
such that $m_n^*(\abs{f-f^n_j}>2^{-j})\le2^{-j}$. Then, 
$\mu^*(\abs{f-f^n_n}>2^{-n})
	\le
[2^{-n}+\norm{\bar\mu-m_n\wedge2^n\lambda}]/\norm{\bar\mu}$, 
which proves that $f$ is $\mu$ measurable. Moreover,
$$
\norm{\bar\mu}\int\abs fd\mu
	=
\lim_k\int(\abs f\wedge k)d\bar\mu
	=
\lim_k\lim_n\int(\abs f\wedge k)d(m_n\wedge2^n\lambda)
	\le
\limsup_n\int\abs f dm_n
	<
\infty
$$
\end{proof}	

We also obtain the following form of the strong law, with 
\begin{equation}
\Prob_*(\lambda)
	=
\big\{\mu\in\Prob(\A):
\mu\ll\lambda\text{ and }\mu(A)=0\text{ if and only if  }\ \lambda(A)=0\big\}
\end{equation}

\begin{corollary}[Koml\'os]
\label{cor slln}
Let $\seqn f$ be a sequence of measurable functions such that
\begin{equation}
\label{L0 bdd}
\lim_{c\to\infty}\sup_{h\in\co(\abs{f_1},\abs{f_2},\ldots)}\lambda^*(h>c)=0
\end{equation}
There exists a sequence $\seqn g$ in $\G f$ and $\mu\in\Prob_*(\lambda)$
such that, for any subsequence $\seqn{g^\alpha}$, the partial sums
\begin{equation}
\label{sums}
S^\alpha_k=\frac{g^\alpha_1+\ldots+g^\alpha_k}{k}
\qquad
k=1,2,\ldots
\end{equation}
form a Cauchy sequence in $L^1(\mu)$.
\end{corollary}

\begin{proof}
By  \cite[Theorem 6.1]{BK} it is possible to find $\nu\in\Prob_*(\lambda)$
such that the set $\co(\abs{f_1},\abs{f_2},\ldots)$ is bounded in $L^1(\nu)$.
We can then apply Corollary \ref{cor komlos sign} and the remarks that follow 
and obtain a sequence $\seqn g$ in $\G f$ which is $\nu$-Cauchy. By that
same reference it is possible to find $\mu\in\Prob_*(\nu)\subset\Prob_*(\lambda)$ 
and a subsequence (still denoted by $\seqn g$ for simplicity) such that 
$\co(\abs{f_1},\abs{f_2},\ldots)$ is bounded in $L^1(\mu)$ and that
\begin{equation}
\label{absolute sum}
\sup_p\int\sum_{n=n_r}^{n_r+p}\abs{g_{n+1}-g_n}d\mu
\le
2^{-r}
\qquad
r\in\N
\end{equation}
for some suitably chosen sequence $\seq nr$. Let $k>n_r$ and 
$S_k=k^{-1}\sum_{n=1}^kg_n$. Then,
\begin{align*}
k\int\dabs{S_k-g_{n_r}}d\mu
	&\le
\int\sum_{n=1}^{n_r}\abs{g_n-g_{n_r}}d\mu
	+
\int\sum_{n=n_r+1}^k\dabs{g_n-g_{n_r}}d\mu\\
	&\le
\int\sum_{n=1}^{n_r}\abs{g_n-g_{n_r}}d\mu
	+
\int\sum_{n=n_r+1}^k\sum_{i=n_r+1}^n\dabs{g_i-g_{i-1}}d\mu\\
	&\le
2n_r\sup_n\norm{f_n}_{L^1(\mu)}+2^{-r}k
\end{align*}
so that
\begin{equation*}
\sup_{p,q}\norm{S_{k+p}-S_{k+q}}_{L^1(\mu)}
\le
4(n_r/k)\sup_n\norm{f_n}_{L^1(\mu)}+2^{-(r-1)}
\end{equation*}
and the sequence $\seq Sk$ is Cauchy in $L^1(\mu)$. It is clear that 
\eqref{absolute sum}, on which our preceding conclusion rests, holds 
for the sequence $\seqn g$ if it holds for all of its subsequences.
\end{proof}

The comparison of Corollary \ref{cor slln} with the original result of Koml\'os
illustrates the difficulties inherent in finite additivity. Not only is the original
property of a.s. convergence replaced here by the Cauchy criterion, but also 
a change of measure is necessary to prove the claim. Given that $\lambda$ 
and $\mu$ have the same null sets, these limitations are irrelevant in the case 
of a countably additive measure. It should be noted, however, that the change 
of measure technique is useful here to relax more familiar integrability conditions 
which are traditionally employed in the proof of the strong law.

In closing, one should mention that a finitely additive version of the strong law 
has been proved long ago, by Chen \cite{chen} and later by Halevy and  
Bhaskara Rao \cite{halevy rao} who also proved a version of Koml\'os theorem. 
Other important papers that follow a similar approach to finitely additive limit 
theorems are those of Karandikar \cite{karandikar} and of Ramakrishnan 
\cite{ramakrishnan}. The setting adopted in these and related papers is however 
that of independent strategies mentioned above and is therefore radically different 
from ours. The connection between this approach and the one proposed in this 
work surely deserves further study.


\begin{thebibliography}{9}
\bibitem{aliprantis burkinshaw}
C. D. Aliprantis, O. Burkinshaw:
\textit{Positive Operators},
Academic Press, Orlando 1985.

\bibitem{balder 89}
E. J. Balder: 
\textit{Infinite-Dimensional Extension of a Theorem of Koml\'{o}s},
Probab. Th. Rel. Fields \textbf{81} (1989), 185-188.

\bibitem{balder}
E. J. Balder: 
\textit{New Sequential Compactness Results for Spaces of Scalarly Integrable Functions}, 
J. Math. Anal. Appl. \textbf{151} (1990), 1-16.

\bibitem{balder hess}
E. J. Balder, C. Hess: 
\textit{Two Generalizations of Koml\'os' Theorem with Lower Closure-Type Applications}, 
J. Convex Anal. \textbf{3} (1996), 25-44.

\bibitem{berkes}
I. Berkes: 
\textit{An Extension of the Koml\'{o}s Subsequence Theorem},
Acta Math. Hung. \textbf{55} (1990), 103-110.

\bibitem{blackwell dubins}
D. Blackwell, L. E. Dubins:
\textit{Merging of Opinions with Increasing Information},
Ann. Math. Stat. \textbf{33} (1962), 882-886.

\bibitem{brooks}
J. K. Brooks: \textit{Weak Compactness in the Space of Vector Measures},
Bull. Amer. Math. Soc. \textbf{78} (1972), 284-287.

\bibitem{brooks dinculeanu}
J. K. Brooks, N. Dinculeanu: 
\textit{Strong Additivity, Absolute Continuity and Compactness in Spaces of Measures}, 
J. Math. Anal. Appl. \textbf{45} (1974), 156-175.

\bibitem{burkholder}
D. L. Burkholder: 
\textit{Discussion on prof. Kingman's Paper}, 
Ann. Probab. \textbf 1 (1973), 900-902.

\bibitem{BK} 
G. Cassese: 
\textit{Convergence in Measure under Finite Additivity}, 
Shanky\= a A, \textbf{75} (2013), 171-193.

\bibitem{JOTP} 
G. Cassese: 
\textit{Finitely Additive Supermartingales}, 
J. Theo. Probab., \textbf{21} (2008), 586-603.

\bibitem{chatterji}
S. D. Chatterji: 
\textit{A General Strong Law}, 
Inventiones Math. \textbf{9} (1970), 235-245.

\bibitem{chen}
R. Chen:
\textit{Some Finitely Additive Versions of the Strong Law of Large Numbers},
Israel J. Math. \textbf{24} (1976), 244-259.

\bibitem{chen as}
R. Chen:
\textit{On Almost Sure Convergence in a Finitely Additive Setting},
Z. Wahrsch. Ver. Geb. \textbf{37} (1977), 341-365.

\bibitem{cvitanic karatzas}
J. Cvitani\' c, I. Karatzas: 
\textit{Generalized Neyman-Pearson Lemma via Convex Duality}, 
Bernoulli \textbf 7 (2001), 79-97.

\bibitem{day lennard}
J. B. Day, C. Lennard:
\textit{Convex Koml\'os Sets in Banach Function Spaces}, 
J. Math. Anal. Appl. \textbf{367} (2010), 129-136.

\bibitem{diestel uhl}
J. Diestel, J. J. Uhl Jr.: 
\textit{Vector Measures}, 
Mathematical Surveys 
\textbf{15}, Amer. Math. Soc., Providence, 1977.

\bibitem{doleans}
C. Dol\'eans-Dade:
\textit{Existence du Processus Croissant Naturel Associ\'e
\`a un Potentiel de la Classe (D)},
Z. Wahrsch. verw. Geb. \textbf{9} (1968), 309-314.

\bibitem{dubins savage}
L. E. Dubins, L. J. Savage:
\textit{How to Gamble if You Must. Inequalities for Stochastic Processes},
Dover, New York, 2014.

\bibitem{bible} 
N. Dunford, J. T. Schwartz: 
\textit{Linear Operators. General Theory}, 
Wiley, New York, 1988.

\bibitem{guessous}
M. Guessous: 
\textit{An Elementary Proof of Koml\'{o}s-R\'ev\'esz Theorem in Hilbert Space}, 
J. Convex Anal. \textbf{4} (1997), 321-332.

\bibitem{halevy rao}
A. Halevy, M. Bhaskara Rao: 
\textit{On an Analogous of Koml\'{o}s's Theorem for Strategies}, 
Ann. Probab. \textbf{7} (1979), 1073-1077.

\bibitem{jimenez}
E. Jim\'enez Fern\'andez, M. A. Juan, E. A. S\'anchez P\'erez:
\textit{A Koml\'os Theorem for Abstract Banach Lattices of Measurable Functions}, 
J. Math. Anal. Appl. \textbf{383} (2011), 130-136.

\bibitem{karandikar}
R. L. Karandikar:
\textit{A general principle for limits theorems in finitely additive probability},
Trans. Amer. Math. Soc. \textbf{273} (1982), 541-550.

\bibitem{komlos} 
J. Koml\'{o}s: 
\textit{A Generalization of a Problem of Steinhaus},
Acta Math. Hung. \textbf{18} (1967), 217-229.

\bibitem{lennard}
C. Lennard: 
\textit{A Converse to a Theorem of Koml\'os for Convex Subsets of $L_1$}, 
Pacific J. Math. \textbf{159} (1993), 75-85.

\bibitem{purves sudderth}
R. A. Purves, W. D. Sudderth: 
\textit{Some Finitely Additive Probability},
Ann. Probab. \textbf 4 (1976), 259-276.

\bibitem{ramakrishnan}
S. Ramakrishnan:
\textit{Finitely Additive Markov Chains},
Trans. Amer. Math. Soc. \textbf{265} (1981), 247-272.

\bibitem{schwartz}
M. Schwartz: 
\textit{New Proofs of a Theorem of Koml\'{o}s},
Acta Math. Hung. \textbf{47} (1986), 181-185.

\bibitem{trautner}
R. Trautner: 
\textit{A New Proof of the Koml\'os-R\'ev\'esz Theorem},
Probab. Th. Rel. Fields \textbf{84} (1990), 281-287.

\bibitem{weizsacker}
H. von Weizs\"acker: 
\textit{Can One Drop $L^1$-Boundedness in Koml\'{o}s's 
Subsequence Theorem?}, Amer. Math. Month. \textbf{111} (2004), 900-903.	
\end{thebibliography}
\end{document}